\documentclass[12pt]{amsart}
\usepackage{color,graphicx,amssymb,latexsym,amsfonts,amsmath,amsthm}
\usepackage{pstcol,pifont,geometry,amsmath}

\usepackage{pdfsync}

\input epsf
\newcommand{\s}{\vspace{0.3cm}}
\newtheorem{theo}{Theorem}
\newtheorem{prop}[theo]{Proposition}

\theoremstyle{remark}
\newtheorem{rema}[theo]{\bf Remark}

\begin{document}
\title{Field of moduli of generalized Fermat curves}
\author{Sebastian Reyes Carocca}
\thanks{Partially supported by project UTFSM 12.09.02}
\address{Departamento de Matem\'atica, Universidad T\'ecnica Federico Santa Mar\'{\i}a, Casilla 110-V Valparaiso, Chile}
\email{sebastian.reyes@usm.cl}

\begin{abstract}
As a consequence of the Riemann-Roch theorem, a closed Riemann
surface $S$ can be described by a non-singular complex projective
algebraic curve $C$. A field of definition for $S$ is any subfield
$D$ of $\mathbb{C}$ so that we may choose $C$ to be defined by
polynomials in $D[x_0, \ldots, x_n]$. The field of moduli of $S$ is
${\mathbb R}$ if and only if $S$ admits an anticonformal
automorphism. In the case that the field of moduli of $S$ is
${\mathbb R}$, then $S$ can be defined over the field of moduli if
and only if $S$ admits an anticonformal involution. It may happen
that the field of moduli is not a field of definition.

In this paper, we consider certain class of closed Riemann surfaces,
called generalized Fermat curves. These surfaces are the highest Abelian
branched cover of certain orbifolds. In this class of Riemann surfaces,
we study the problem of deciding when the field of moduli is ${\mathbb R}$
and when, in such a case, it is a field of definition.
\end{abstract}

\maketitle

%%%%%%%%%%%%%%%%%%%%
\section{Introduction}
\vspace{0,5 cm} Let $S$ be a closed Riemann surface. As a
consequence of the Riemann-Roch theorem \cite{Farkas-Kra}, $S$ can
be described by a non-singular complex projective algebraic curve $C
\subset {\mathbb P}^{n}(\mathbb{C})$. If $C$ can be chosen to be
defined by homogeneous polynomials $f_{1},\ldots,f_{r} \in
D[x_{0},\ldots,x_{n}]$, where $D$ is a subfield of ${\mathbb C}$,
then we say that $D$ is a field of definition of $S$.

Let $C$ be defined by the homogeneous polynomials
$f_{1},\ldots,f_{r} \in {\mathbb C}[x_{0},\ldots,x_{n}]$, then the
complex conjugated curve $\overline{C}$ is the algebraic curve
defined by the homogeneous polynomials
$\widehat{f_{1}},\ldots,\widehat{f_{r}} \in {\mathbb
C}[x_{0},\ldots,x_{n}]$, where $\widehat{f_{j}}$ is obtained by
application of $\sigma(z)=\overline{z}$ to the coefficients of
$f_{j}$. The field of moduli of $S$ is then defined as ${\rm
M}(S)={\mathbb R}$ if $C$ and $\overline{C}$ are conformally
equivalent as closed Riemann surfaces. Note that the above
definition of ${\rm M}(S)$ does not depends on the choice of $C$. If
$J_{n}:{\mathbb P}^{n}(\mathbb{C}) \to {\mathbb P}^{n}(\mathbb{C})$
is the conjugation $J_n([x_{0} : \ldots : x_{n}])=[\overline{x_{0}}:
\ldots :\overline{x_{n}}]$, then $J_n:C \to \overline{C}$ defines an
anticonformal isomorphim (as closed Riemann surfaces). In this way,
the field of moduli of $S$ is ${\mathbb R}$ if and only if $C$
admits an anticonformal automorphism

If $S$ can be defined over ${\mathbb R}$, we also say that $S$ is a
real Riemann surface, then we may chose $C$ defined by polynomials
$f_{1}, \ldots, f_{r} \in {\mathbb R}[x_{0},\ldots,x_{n}]$. In this case,
$J_{n}$ defines an anticonformal involution on $S$. Conversely, as a
consequence of Weil's theorem \cite{Silhol, Weil}, if $S$ admits an
anticonformal involution, then $S$ can be defined over ${\mathbb
R}$.

We are interested on those closed Riemann surfaces whose field of
moduli is ${\mathbb R}$ and which can or cannot be definable over
${\mathbb R}$.

By the uniformization theorem, there is one conformal class of
Riemann surfaces of genus $0$; this given by the Riemann sphere
$\widehat{\mathbb C}$. This clearly has an anticonformal involution
($J(z)=\overline{z}$); it follows that a genus zero Riemann surface
has field of moduli equal to ${\mathbb R}$ and that it is real. A
closed Riemann surface of genus one can be described by an algebraic
curve of the form $C_{\lambda}:=\{y^{2}z=x(x-z)(x-\lambda z)\}
\subset {\mathbb P}^{2}(\mathbb{C})$, where $\lambda \in {\mathbb
C}-\{0,1\}$. If
$j(\lambda)=(1-\lambda+\lambda^2)^3/\lambda^2(\lambda-1)^{2}$ is its
$j$-invariant and $a(\lambda)=27j(\lambda)/(j(\lambda)-1)$, then
$C_{\lambda}$ is isomorphic to
$D_{\lambda}=\{y^{2}=4x^3-a(\lambda)x-a(\lambda)\}$; so ${\mathbb
Q}(j(\lambda))$ is a field of definition of $C_{\lambda}$. It can be
seen that $C_{\lambda}$ is real if and only if $j(\lambda)$ is real.

If $S$ has genus $g \geq 2$, then the situation gets more
complicate. The first examples of closed Riemann surfaces of genus
at least two which are not real and whose field of moduli is
${\mathbb R}$ where provided by Shimura \cite{Shimura} and Earle
\cite{Earle} around 1972. These examples where all hyperelliptic
Riemann surfaces (that is, there is a two-fold branched cover over
Riemann sphere). More recently, in \cite{Hidalgo} a
non-hyperelliptic non-real curve with field of moduli equal to
${\mathbb R}$ was provided. Such a non-hyperelliptic example
(depending on two real parameters) turns out to be the homology
cover of an orbifold with signature $(0;2,2,2,2,2,2)$, that is, a
closed Riemann surface $S$ of genus $17$ admitting a group $H \cong
{\mathbb Z}_{2}^{5}$ as a group of conformal automorphisms so that
$S/H$ is the Riemann sphere with exactly $6$ cone points, each one
of order $2$.

In this paper we consider those closed Riemann surfaces $S$
admitting a group $H \cong {\mathbb Z}_{p}^{n}$, where $p \geq 2$ is
a prime and $n \geq 2$, so that $S/H$ is an orbifold with signature
$(0;p,\stackrel{n+1}{\ldots},p)$. We study the problem of deciding
when such a surfaces have field of moduli equal to ${\mathbb R}$ and
when they are reals. \vspace{0,5 cm}

%%%%%%%%%%%%%%%%%
%%%%%%%%%%%%%%%%%
\section{Preliminaries}
\subsection{Riemann orbifolds and Fuchsian groups}

Let $S$ a closed Riemann surface. We will denote by $\mbox{Aut}(S)$
its full group of conformal automorphims. For $H$ subgroup of
$\mbox{Aut}(S)$, we denote by $\mbox{Aut}_H(S)$ the normalizer of
$H$ inside $\mbox{Aut}(S)$ and by $H'$ its conmutator subgroup.

A {\it Riemann orbifold} ${\mathcal O}$ of {\it signature}
$s({\mathcal O})=(\gamma : m_1,\ldots, m_r)$ is given by a closed
Riemann surface $S$ of genus $\gamma$ (called the underlying Riemann
surface structure of ${\mathcal O}$), a collection of $r$ different
points, say $p_{1},..., p_{r} \in S$ (called the cone points) and an
assignation of an integer $m_{j} \geq 2$ to the point $p_{j}$
(called the cone order of the cone point $p_{j}$). By a conformal
automorphism of a Riemann orbifold $\mathcal{O}$ we mean a conformal
automorphism of the underlying Riemann surface that preserve the
conic points and their orders. We denote by
$\mbox{Aut}_{Orb}(\mathcal{O})$ the conformal automophisms group of
$\mathcal{O}$.

By a \textit{fuchsian group} we mean a discrete subgroup of the
group $\mbox{Aut}(\mathbb{H}^{2}) \cong \mbox{PSL}(2, \mathbb{R})$ of
conformal automorphisms of the upper half plane. For details see
\cite{Maskit}. If $\Gamma$ is a co-compact fuchsian group, then
$\Gamma$ has a presentation in terms of $2\gamma$ hyperbolic
generators, say, $a_1,b_1, \ldots, a_{\gamma}, b_{\gamma}$ and $r$
eliptics, say $x_1,\ldots,x_r$, with the relations $$
[a_1,b_1]\cdots [a_{\gamma}, b_{\gamma}] \cdot x_1\cdots
x_r=x_1^{m_1}=\cdots =x_r^{m_r}=1
$$where $[a,b]=aba^{-1}b^{-1}$.  The {\it signature} of $\Gamma$ in this case is given by
$ s(\Gamma)=(\gamma : m_1,\ldots, m_r)$. In this case,
the quotient $\mathcal{O}=\mathbb{H}/\Gamma$ is a Riemann orfbifold with
$s({\mathcal O})=s(\Gamma)$.

By the classical uniformization theorem \cite{Farkas-Kra}, every
compact Riemann surface $S$ genus $g \geq 2$, can be realized as a
quotient $\mathbb{H}/\Gamma$ of the hyperbolic plane $\mathbb{H}$
under the action of a torsion free co-compact fuchsian group
$\Gamma$. We set $s(\Gamma)=(g: - )$ and say that $\Gamma$ is a
surface group.

A finite asbtract group $G$ acts as a group of automorphisms of
$S=\mathbb{H}/\Gamma$ if and only if, $G \cong \Lambda/\Gamma$ for
some fuchsian group $\Lambda$ that contains $\Gamma$ as a normal
subgroup index $|G|$; equivalently, if there exists an epimorphism
of groups $\Theta : \Lambda \longrightarrow G$  with $\Gamma =
\mbox{ker}(\Theta)$.

%%%%%%%%%%%%%%%%%
\subsection{Generalized Fermat curves}
Let $n,k \ge 2$ integers. A closed Riemann surface $S$ is called a
generalized Fermat curve of type $(k,n)$ if exists a subgroup $H <
\mbox{Aut}(S)$, $H \cong \mathbb{Z}_k^n$ (direct sum of $n$ copies
of $\mathbb{Z}_k$) so that $S/H$ is a Riemann orbifold of signature
$(0;k,\stackrel{n+1}{\ldots},k)$. We say that $H$ is a generalized
Fermat group of type $(k,n)$ and the pair $(S,H)$ is a generalized
Fermat pair of type $(k,n)$. By the Riemann-Hurwitz formula
\cite{Farkas-Kra}, the genus of a generalized Fermat curve of type
$(k,n)$ is
$$g(k,n)=1+ \frac{k^{n-1}}{2}((n-1)(k-1)-2).$$

Two pairs $(S_1,H_1)$ and $(S_2,H_2)$ of the same type are
topologically equivalent (conformally equivalent) if exists an
orientation-preserving homeomorphism (conformal homeomorphism)
$\varphi : S_1 \rightarrow S_2$ so that $\varphi^{-1} H_2
\varphi=H_1$.

The only non-hyperbolic generalized Fermat pairs are of type
$(2,2)$, $(2,3)$ and $(3,2)$. For example, if $(S,H)$ is a
generalized Fermat curve of type $(2,2)$ then $S$ is genus zero,
therefore $S$ is conformally equivalent to the Riemann sphere and
the generalized Fermat group $H \cong \mathbb{Z}_2 \oplus
\mathbb{Z}_2$ is generated for the transformations $z \mapsto -z$
and $z \mapsto 1/z$.

For hyperbolic cases, clasical uniformization theorem asserts that
the Riemann orbifold $S/H$ is uniformizated by a fuchsian group
$\Gamma < \mbox{Aut}(\mathbb{H}^2)$ whose presentation is $$ \Gamma
= \langle x_1, \ldots, x_{n+1} : x_1^k = \ldots = x_{n+1}^k = x_1
\cdots x_{n+1} =1 \rangle.$$

\s
\noindent
\begin{prop}[\cite{Fermat}]
Let $(S,H)$ a hyperbolic generalized Fermat curve of type $(k,n)$
and let $\Gamma$ be a orbifold univeral cover group of $S/H$. Then
$S$ is non-hyperelliptic and $(S,H)$ and  $(\mathbb{H}^2/  \Gamma',
\Gamma / \Gamma')$ are conformally equivalent generalized Fermat
pairs.
\end{prop}

\s

Let us consider a generalized Fermat pair $(S,H)$ of type $(k,n)$.
Let us assume, after a M\"obius transformation, that the cone locus
of the orbifold $S/H$ is given by
$$\infty, 0, 1, \lambda_{1},\ldots, \lambda_{n-2}.$$

Consider the non-singular projective algebraic curve in
$\mathbb{P}^n(\mathbb{C})$ defined by
$$C(\lambda_1,\ldots,\lambda_{n-2}; k) :  \left \{ \begin{array}{lllllll}
x_1^k & + & x_2^k & + & x_3^k & =  & 0\\
\lambda_1 x_1^k & + & x_2^k & + & x_4^k & =  & 0\\
\lambda_2 x_1^k & + & x_2^k & + & x_5^k & =  & 0\\
\vdots & \; & \vdots & \; &  \vdots  &  \; & \vdots \\
\lambda_{n-2} x_1^k & + & x_2^k & + & x_{n+1}^k & =  & 0\\\end{array} \right\}.$$

As $\lambda_j \in \mathbb{C}-\{0,1\}$, $\lambda_i \neq \lambda_j$ if
$i \neq j$, it can seen that $C(\lambda_1,\ldots,\lambda_{n-2}; k)$
is a non-singular algebraic curve; so it describes a closed Riemann
surface. We note that the projective linear transformations
$$a_j([x_1: \cdots :x_{n+1}])  = [x_1:\cdots:
x_{j-1}:\exp\{2\pi i/k\} x_j:x_{j+1}:\cdots:x_n : x_{n+1}], $$ with
$j \in \{ 1,\cdots,n \}$ provides a faithful representation
$$\Theta : \mathbb{Z}_k^n \hookrightarrow
\mbox{Aut}(C(\lambda_1,\ldots,\lambda_{n-2}; k)).$$In addition, the
$k^n$ degree conformal map
$$\pi :C(\lambda_1,\ldots,\lambda_{n-2} ; k) \rightarrow \widehat{\mathbb{C}}$$
defined by $\pi ([x_1:\cdots:x_{n+1}])= - (x_2/x_1)^k$, satisfies $
\pi \circ a_j = \pi$ for every $1\leq j \leq n$. We note that
$H_0:=\Theta(\mathbb{Z}_k^n)$ is a generalized Fermat group of type
$(k,n)$ for the closed Riemann surface
$C(\lambda_1,\ldots,\lambda_{n-2}; k)$.

\s
\noindent
\begin{prop}[\cite{Fermat}]
The generalized Fermat pairs
$(S,H)$ and  $(C(\lambda_1,\ldots,\lambda_{n-2} ; k) , H_{0})$
are conformally equivalent.
\end{prop}

\s

We have that $ \mbox{Aut}_{H_0}(C(\lambda_1,\ldots,\lambda_{n-2} ; k))/H_0
 $ is a group isomorphic to the subgroup
of $\mbox{PSL}(2, \mathbb{C})$ that preserves the conic set $\{
\infty, 0, 1, \lambda_1,\ldots,\lambda_{n-2} \}$. It follow that, if
$H_0$ is unique (or normal) inside $\mbox{Aut}(S)$, then

$$
\frac{ \mbox{Aut}(C(\lambda_1,\ldots,\lambda_{n-2} ; k))}{H_0 }
\cong \mbox{Aut}_{Orb}(S/H_0)
$$

\s
\noindent
\begin{rema} \label{notalift}
We note that if $H_0$ is normal, then we can obtain $\mbox{Aut}(S)$
lifting $\mbox{Aut}_{Orb}(\mathcal{O})$. In fact, if $f \in
\mbox{Aut}_{Orb}(S/H)$, then exists $\widehat{f} \in \mbox{Aut}(S)$
so that $\pi \widehat{f} = f \pi$. In \cite{Fermat} they note that
if $f$ induces the permutation $\sigma \in \mathbf{S}_{n+1}$ in the
set $\{\mu_1= \infty, \mu_2=0, \mu_3=1, \mu_4 =\lambda_1, \ldots,
\mu_{n+1} = \lambda_{n-2}\}$, then $\widehat{f}$ is defined by $$
\widehat{f}([x_1: \ldots : x_{n+1}])=[c_1x_{\sigma^{-1}(1)}: \ldots
: c_{n+1}x_{\sigma^{-1}(n+1)} ] $$where the complex constants $c_j$
can be easily computed using the algebraic equations that defines
the curve. (for simplicty, we may assume $c_1=1$).
\end{rema}

\s

The next results will be important in the proofs of the main results
of this paper.

\s
\noindent
\begin{theo}[\cite{humbert}]\label{(2,4)}
Let $(S,H)$ a generalized Fermat pair of type $(2,4)$. Then $H$ in
unique inside $\mbox{Aut}(S)$.
\end{theo}

\s
\noindent
\begin{theo}[\cite{Fermat}]\label{unico}
Let $S$ a generalized Fermat curve of type $(p,n)$ with $p\ge 2$
prime and $n \ge 2$ so that $(n-1)(p-1)>2$. If $H_1$ and $H_2$ are
two generalized Fermat groups of type $(p,n)$ then they are
conjugate inside $\mbox{Aut}(S)$.
\end{theo}

\s
%%%%%%%%%%%%%%%%%
\subsection{Action of the Galois group} Let $F < E$ an extension fields
and consider the Galois group $\mbox{Gal}(E/F)$ asociated to the extension.
$\mbox{Gal}(E/F)$ acts in the polynomial ring $E[x_0 , \ldots , x_n]$:
if
$$f(x_0 , \ldots, x_n)=  \sum a_{i_0 \cdots i_n}x_0^{i_0}\cdots x_n^{i_n}$$
then $$(\sigma \cdot f )(x_0 , \ldots, x_n)= \sum \sigma \left( a_{i_0 \cdots i_n} \right)       x_0^{i_0}\cdots x_n^{i_n}
$$Set $f^{\sigma}:=\sigma \cdot f$. This accion induces an action in
the set of projective algebraic varieties. If $X$ is defined by $f_1
, \ldots, f_r$, then we can consider the polynomials $f_1^{\sigma} ,
\ldots, f_r^{\sigma}$ that defines a new projective algebraic
variety; say $X^{\sigma}$.

%%%%%%%%%%%%%%%%%
\subsection{Field of Moduli and Fields of Definition}
Let $F<E$ an extension fields and let $X \subset \mathbb{P}^n(E)$
be a projective algebraic variety.

The field of moduli $\mbox{M}_{E/F}(X)$ of $X$, asociated to the
extension $F<E$, is defined as the fixed field of the subgroup
$$E_F(X)= \{\, \sigma \in \mbox{Gal}(E/F): X \simeq X^{\sigma} \},$$
where ``$\simeq$" means birational isomorphism.
It is clear from the definition that $F< M_{E/F}(X)<E$.

A field of definition for $X$ is any subfield $D$, $F<D<E$, so that
exists $Y \simeq X$ defined by polynomials in $D[x_0, \ldots, x_n]$.
It is clear from the definition that if $D$ is a field of definition
for $X$, then every extension of $D$ (inside $E$) is also a field of
definition, nevertheless, it is not clear that there is a  smallest
field of definition.

If $F < E$ is a general Galois extension (i.e. for every $F < N < E$
holds that $\mbox{Fix}(\mbox{Gal}(E/N)=N$), then is known that the
field of moduli is contained in every field of definition. We are
interested in the Galois extension $\mathbb{R} < \mathbb{C}$. As a
Galois extension is  a general Galois extension, the previuos result
holds. The main result here is a theorem due by A. Weil in 1956, see
\cite{Weil}. We present a simplicated version of this theorem
(sufficient for our porpouse).

\s
\noindent
\begin{theo}[Weil's Theorem]
Let $F < E$ a finite Galois extension, $X \subset \mathbb{P}^n(E)$
be a projective algebraic variety and $F=M_{E/F}(X)$. Then $F$ is a
field of definition for $X$ if and only if for every $\sigma \in
E_F(X)= \mbox{Gal}(E/F)$ there exists  a birational isomorphism
$f_{\sigma} : X \rightarrow X^{\sigma}$ so that for every $ \sigma,
\tau \in \mbox{Gal}(E/F)$ it holds that $f_{\sigma
\tau}=f_{\tau}^{\sigma} \circ f_{\sigma}$.
\end{theo}

\s

%%%%%%%%%%%%%%%%%%%%%%%%

\subsection{The complex case}
We are interested in the complex case, that is,  $E=\mathbb{C}$ and
$F= \mathbb{R}$; a Galois extension of degree two. In this case, the
projective algebraic variety $X$ became in complex algebraic
variety. We will interested only in the case when $X$ is a
non-singular curve (that is a closed  Riemann surface).

We know that the  field of moduli is contained in every field of
definition, in particular, the intersection of all the fields of
definition contains the moduli field. An interesting question is to
know when the  field of moduli is a field of definition and when it
is $\mathbb{R}$ (this is equivalent to define $X$ using real
polynomials).

The  field of moduli is $\mathbb{R}$ if and only if $X$ and
$X^{\sigma}$ are birrational equivalent for every $\sigma \in
\mbox{Gal}(\mathbb{C}/ \mathbb{R})=\{id, z \mapsto \overline{z} \}$.
It is clear that the  field of moduli is ${\mathbb R}$ if and only
if $X$ admits an anticonformal automorphism. For the other hand, if
$X$ can be defined by real polynomial, then $J_n :
\mathbb{P}^n(\mathbb{C}) \longrightarrow \mathbb{P}^n(\mathbb{C})$
defined by  $J([x_{0},\ldots,x_{n}])=[\overline{x_{0}},
\ldots,\overline{x_{n}}]$, induces an order two anticonformal
automorphism for $X$. Reciprocally, if $X$ admits and order two
anticonformal automorphism $\tau$, it is not dificult to prove that
$\{f_e=id, f_{\sigma} =\tau \circ J_n \}$ satisfices the Weil
Theorem conditions.

%%%%%%%%%%%%%%%%%
\subsection{Examples: Curves}
We have already noted the cases when the Riemann surfaces are of
genus $0$ and $1$. For genus $g \ge 2$ the problem is more
complicated. Shimura and Earle provided the first examples of
algebraic curves with  field of moduli ${\mathbb R}$ but which are
not definable using real polynomials. We proceed to recall these
examples.

\s
\begin{enumerate}

\item[(1)] \it{ \bf Earle's example} \cite{Earle}: Let $a \in ( - \infty, -(3+ \sqrt{2}) )$ and $b \in \mathbb{H}^2$ with $|b|^2=-a$. Consider $$  p(x,y,z)= y^2z^3-x(x-z)(x-az)(x^2-b^2z^2) $$and $X \subset \mathbb{P}^2(\mathbb{C})$ the hyperelliptic algebraic curve defined by $p$. Then the field of moduli of $X$ is ${\mathbb R}$ and it cannot be defined over ${\mathbb R}$.

\s

\item[(2)] \it{\bf Shimura's example} \cite{Shimura}: Let $a_0 \in \mathbb{R}$, $a_m=1,\, a_1, a_2, \ldots, a_{m-1} \in \mathbb{C} $ and $m$ an odd positive integer so that the set $$C= \{ a_0, a_1, \ldots, a_{m-1}, \overline{a_1}, \ldots, \overline{a_{m-1}} \}  $$will be algebraically independient over $\mathbb{Q}$. We let consider$$p(x,y,z)= y^2z^{2m-2}-a_0x^mz^m - \sum_{j=1}^m \left(  a_jx^{m+j}x^{m-j}+(-1)^j \overline{a_j}x^{m-j}z^{m+j}  \right)$$
and $X \subset \mathbb{P}^2(\mathbb{C}) $ the plane algebraic curve defined by $p$. Then the field of moduli of $X$ is ${\mathbb R}$ and it cannot be defined over ${\mathbb R}$.
\end{enumerate}

For details and explicit anticonformal automorphisms see \cite{Earle} and \cite{Shimura}. The Shimura and Earle examples are hyperelliptic algebraic curves. The first non-hyperelliptic example was due by R. Hidalgo en \cite{Hidalgo}.

\s

\begin{enumerate}
\item[(3)] \it{\bf Hidalgo's example} \cite{Hidalgo}:
Let $\lambda_1 \in \mathbb{R}$,
$\lambda_2 \in \mathbb{C}$ so that $\lambda_1<-(3+2\sqrt{2})$,
$\mbox{Im}(\lambda_2)>0$, $\mbox{Re}(\lambda_2)<0$,
$|\lambda_2|^2=-\lambda_1$ and $X$ the non-singular projective algebraic curve defined by
$$ X :   \left \{ \begin{array}{lllllll}
x_1^2 & + & x_2^2 & + & x_3^2 & =  & 0\\
\lambda_1 x_1^2 & + & x_2^2 & + & x_4^2 & =  & 0\\
\lambda_2 x_1^2 & + & x_2^2 & + & x_5^2 & =  & 0\\
- \lambda_{2} x_1^2 & + & x_2^2 & + & x_{6}^2 & =  & 0\\\end{array} \right\} $$
in $\mathbb{P}^5(\mathbb{C})$. Then $X$ is a non-hyperelliptic closed Riemann surface genus 17 which admits an anticonformal automorphism order 4 but does not admit an anticonformal involution. In particular, the field of moduli of $X$ is ${\mathbb R}$ but it cannot be definable over ${\mathbb R}$.

\end{enumerate}

We note that Hidalgo's example is a generalized Fermat curve type
$(2,5)$. An important fact in this example   is the uniqueness of
the generalized Fermat group inside the conformal automorphisms full
group.

Next result will be frequently used in the proofs.

\s
\noindent
\begin{prop}\label{involucion}
Let $k \geq 3$ be an odd integer and $(S,H)$ be a generalized Fermat
pair of type $(k,n)$, where $n \geq 2$. If the orbifold $S/H$
admits an anticonformal involution, then $S$ has field of moduli
${\mathbb R}$ and it is a field of definition.
\end{prop}
\begin{proof}
Let $\widehat{\tau}:S/H \to S/H$ be an anticonformal involution.
Then, as $S$ is the homology cover of $S/H$, the involution
$\widehat{\tau}$ lifts as an anticonformal automorphism $\tau:S \to
S$. So the field of moduli of $S$ is ${\mathbb R}$. As
$\widehat{\tau}$ has order two, it follows that $\tau^{2} \in H$. As
$H$ has odd order, then $\tau^{2}$ is either the identity or it has
odd order, say $s$. In the last case, $\tau^{s}$ is an anticonformal
involution. The existence of an anticonformal involution is
equivalent for $S$ to be real.
\end{proof}

%%%%%%%%%%%%%%%%%
%%%%%%%%%%%%%%%%
\section{Main results}
In this section, we will always consider the extension $\mathbb{R} <
\mathbb{C}$ and we set $\mbox{M}(X)$ instead of
$\mbox{M}_{\mathbb{C}/ \mathbb{R}}(X)$ for simplicity.

Let $(S,H)$ a generalized Fermat pais of type $(k,n)$ and let
$\widehat{\tau}$ be an anticonformal automorphism (as orbifold) of
$S/H$ of even order $2M$ with $M \ge 1$. Let $P \subset
\widehat{\mathbb C}$ be the set of cone points of the Riemann
orbifold $S/H$. As $\widehat{\tau}$ must  keep invariant $P$ and $P$
has cardinality $n+1$,  we may assume (after conjugation by a
suitable M\"obius transformation) that
$\widehat{\tau}(z)=\mbox{e}^{i\theta}/\overline{z}$ with $\theta = 2
\pi /N$, (rotation and reflection in $\mathbf{S}^1$). In this case,
if $N$ is odd, then $N=M$ and if $N$ is even, then $2M=N$.

We consider the action of the cyclic group $\langle
\widehat{\tau}\rangle \cong {\mathbb Z}_{2M}$ over the set $P$.
\begin{enumerate}
\item If $N \geq 3$ is odd, then we have exactly $3$ possible types of orbits: $C \in \{0,1\}$ orbits of length $2$, $B$ orbits of length $N$ and $A$ of length $2N$.

\item If $N \geq 4$ is even, then we have exactly $2$ possible types of orbits: $C \in \{0,1\}$ orbits of length $2$, and $B$ orbits of length $N$. In this case we set $A=0$.

\item If $N=1$, then we have exactly $2$ possible types of orbits: $A$ orbits of length $2$ and $B$ orbits of length $1$. In this case we set $C=0$.

\item If $N=2$, then we have  $B$ orbits of length $2$. In this case we set $A=C=0$.

\end{enumerate}

It is clear from the definition of $A$, $B$ and $C$ that
$$(*) \quad n+1= 2NA + NB + 2C.$$

We will make use of this in the rest of the paper.

\subsection{}
Our first result concerns Hidalgo's example.
\begin{theo}
Let $(S,H)$ a generalized Fermat pair of type $(k,5)$ with $k \ge 2$
of the form
 $$C_k :  \left \{ \begin{array}{lllllll}
\,\,\,\,\,\,\,x_1^k & + & x_2^k & + & x_3^k & =  & 0\\
\,\,\,\,\lambda_1 x_1^k & + & x_2^k & + & x_4^k & =  & 0\\
\,\,\,\, \lambda_2 x_1^k & + & x_2^k & + & x_5^k & =  & 0\\
- \lambda_2 x_1^k & + & x_2^k & + & x_6^k & =  & 0\\\end{array}
\right\}
$$
with $\lambda_1 \in \mathbb{R}$ and $\lambda_2 \in \mathbb{C}$ so that $\lambda_1 = - | \lambda_2 |^2$.
Then the field of moduli of $C_{k}$ is ${\mathbb R}$. Moreover, if $k$ is odd then $C_k$ is real.
\end{theo}

\begin{proof}
Let us consider
the anticonformal automorphism, of order two, $f(z)= \lambda_1 / \overline{z}$ of
$S/H$. We note that $f$ defines the permutation $\sigma = (1 \, 2)(3
\, 4)(5 \, 6)$ in the cone points locus. It follows that, see Remak \ref{notalift},  the lift of
$f$ is of the form
$$\widehat{f}([x_1:x_2:x_3:x_4:x_5:x_6])=[\overline{x_2}: c_2
\overline{x_1}: c_3 \overline{x_4}: c_4 \overline{x_3}: c_5
\overline{x_6}: c_6 \overline{x_5}],$$
for suitable values of $c_{j}'s$. Using the algebraic equations
that defines $S$ is not difficult to see that $c_2^k=\lambda_1$,
$c_3^k=1$, $c_4^k=\lambda_1$, $c_5^k= \lambda_2$ and
$c_6^k=-\lambda_2$. 
We have that $\widehat{f}$ is anticonformal automorphism of $C_k$, so the field of moduli of $C_k$ is ${\mathbb R}$.

Moreover, if $k$ is odd, then we may
choose $c_2=c_4$, $c_3=1$ and $c_5 = -c_6$ and then 
$\widehat{f}^k$ is anticonformal involution of
$C_k$, and then it is real.
\end{proof}

\s

\subsection{}
Our second result states that classical Humbert curves with field of moduli ${\mathbb R}$ are necessarily reals.

\s
\noindent
\begin{theo}
Let $S$ a generalized Fermat curve of type $(2,4)$ (a classical Humbert curve). If the field of moduli is ${\mathbb R}$, then it is a field of definition.
\end{theo}
\begin{proof}
Let $H\cong \mathbb{Z}_2^4$ the generalized Fermat group. Since $H$ is unique inside $\mbox{Aut}(S)$ (see Theorem \ref{(2,4)}) we have that $\mbox{Aut}(S)/H
\cong \mbox{Aut}_{Orb}(S/H)$. Since we are assuming that $M(S)=
\mathbb{R}$, there exists an anticonformal automorphism of $S$, say $\tau : S \rightarrow S$.

Let us denote by $\widehat{\tau} : S/H \rightarrow S/H$
the anticonformal automorphisms (as orbifold) induced by $\tau$. We suppose that $| \widehat{\tau}
|=2M$ for some $M \geq 1$. As already noted, the set of cone points of $S/H$, say $\mu_{1}$, $\mu_{2}$, $\mu_{3}$, $\mu_{4}$ and $\mu_{5}$, should be invariant under $\widehat{\tau}(z)=\mbox{e}^{i\theta}/\overline{z}$ with $\theta = 2 \pi /N$, where $N=M$ for $N$ odd and $N=2M$ for $N$ even.

As $n=4$, it follows from ($*$) that $2NA+NB+2C$ is odd, so $N$ is necessarily odd (that is, $N=M$).  We claim that $N \leq 5$. In fact, if $N \ge 7$ is odd, then $5=2NA + NB + 2C \ge 14A+7B +2C$. It follows that $A=B=0$. But in this case,  $5=2C \in \{0,2\}$, a contradiction.

If$N=1$, then $C=0$ and $2A + B = 5$; so $(A, B,C) \in \{ (0 , 5,0), (1, 3,0), (2, 1,0) \}$.

If $N=3$, then $6A + 3B + 2C = 5$, and this implies $(A, B, C) = (0,1,1)$.

If $N=5$, then $10A+5B+2C=5$, and this implies that $(A, B, C) = (0,1,0)$.

We sumarize the information in the following table:

\vspace{0,5 cm}

\begin{center}
\begin{tabular}{ |l | c| c | r | }
\hline N & A & B & C \\
\hline 1 & 0 & 5 & 0 \\
\hline 1 & 1 & 3 & 0 \\
\hline 1 & 2 & 1 & 0 \\
\hline 3 & 0 & 1 & 1 \\
\hline 5 & 0 & 1 & 0 \\
\hline
\end{tabular}
\end{center}
\vspace{0,5 cm}

The rest of the proof is devoted to analyze separately each case.\vspace{0,5 cm}
\begin{enumerate}
\item[(A)] Case $(N,A,B,C)=(1,0,5,0)$. In this case, all the cone points belong to the unit circle.  Using an appropriate Mobius transformation, we can side the cone points on the real axis; that is, we may assume the cone points to be $\infty$, $0$, $1$ and $\lambda_{1}, \lambda_{2} \in {\mathbb R}$. In this way, the algebraic equations that defines the curve $S$, say $C(\lambda_{1},\lambda_{2})$, is real. In this way, $\mathbb{R}$ is a field of definition for $S$.
\vspace{0,5 cm}
\item[(B)] Case $(N,A,B,C)=(1,1,3,0)$. Conjugation by a M\"obius transformation that keeps invariant the unit disc, we may assume that the cone points are given by $\mu_{1}=\infty$, $\mu_{2}=0$,
$\mu_{3}=1$, $\mu_{4}$ and $\mu_{5}$ with  $| \mu_{4}|=| \mu_{5}|=1$. We have that
$\widehat{\tau}(z)=1/\overline{z}$ induces the permutation $\sigma = (1\,2)$ in the conic point set. We will return to this case below.
\vspace{0,5 cm}
\item[(C)] Case $(N,A,B,C)=(1,2,1,0)$. Conjugation by a M\"obius transformation that keeps invariant the unit disc, we may assume that $\mu_1=\infty$,  $\mu_2 =0$, $\mu_{3}=1$, $\mu_4$ and  $\mu_5=1/\overline{\mu_{4}}$, with $0<| \mu_4|<1$. In this case
$\widehat{\tau}(z)=1/\overline{z}$, and this induce the permutation $\sigma = (1 \, 2)(4 \, 5)$. We will return to this case below.
\vspace{0,5 cm}
\item[(D)] Case $(N,A,B,C)=(3,0,1,1)$. Conjugation by a M\"obius transformation that keeps invariant the unit disc, we may assume that $\mu_1=\infty$, $\mu_2 =0$, $\mu_3= 1$, $\mu_4 = \omega$, and $\mu_5 = \omega^2$ with $\omega=e^{2 \pi i/3}$. In this case,
$\widehat{\tau}(z)=\omega / \overline{z}$, but this configuration also admits the anticonformal involution $z \mapsto 1/\overline{z}$, which induces the permutation $\sigma = (1\,2)$. Observe that  this case is a particular case of $(B)$.
\vspace{0,5 cm}
\item[(E)] Case $(N,A,B,C)=(5,0,1,0)$. In this case the five cone points belong to the unit circle.
As done in the case (A), we have that $\mathbb{R}$ is a field of definition for $S$.
\end{enumerate}
\vspace{0,5 cm}
Now, we analyze the cases (B) and (C).
\begin{enumerate}
\item Case (B). As $\tau:S \to S$ is a lifting of  $\widehat{\tau}(z)=
\frac{1}{\overline{z}} \in {\rm Aut}_{orb}(S/H)$ and the induced permutation is $\sigma = (1\,2)$, it follows that $\tau$ must have the following form (see Remark \ref{notalift})
$$\tau([x_1:x_2:x_3:x_4:x_5])=
[\overline{x_2}:\overline{x_1}:c_3\overline{x_3}:c_4\overline{x_4}:c_5\overline{x_5}],
$$
where $c_3^2=1$, $c_4^2=\lambda_1$ and $c_5^2=\lambda_2$. We can observe that
$\tau$ is an order two anticonformal automorphism of $S$, in particular, that $S$ is real.

\item Case (C). Again, as $\tau:S \to S$ is a lifting of  $\widehat{\tau}(z)=
\frac{1}{\overline{z}} \in {\rm Aut}_{orb}(S/H)$ and the induced permutation is $\sigma = (1 \, 2)(4 \, 5)$, it follows that $\tau$ must have the following form (see Remark \ref{notalift})
$$
\tau([x_1:x_2:x_3:x_4:x_5])=
[\overline{x_2}:\overline{x_1}:c_3\overline{x_3}:c_4\overline{x_5}:c_5\overline{x_4}],
$$
where $c_3^2=1$, $c_4^2=\lambda_1$ and $c_5^2=\lambda_2=1/\overline{\lambda_1}$. We choose $c_4, c_5  \in \mathbb{C}$ so that $c_4c_5=1$.
We can observe that $\tau$ is an order two anticonformal automorphism of $S$, in particular, that $S$ is real.

\end{enumerate}

\end{proof}

\subsection{}

If $(S,H)$ is a hyperbolic generalized Fermat pair of type $(p,n)$ with $p,n \ge 2$ and $p$ prime, then the generalized Fermat group $H$ is unique up to conjugation inside $\mbox{Aut}(S)$ (see Theorem \ref{unico}). This fact allows us to obtain the following.

\vspace{0,5 cm}

\begin{theo}\label{primo1}
Let $S$ a generalized Fermat curve of type
$(p,n)$ with $p \geq 3$ prime and $n \geq 2$ an even integer. If the field of moduli of $S$ is ${\mathbb R}$, then $S$ is real.
\end{theo}

\begin{proof}
Let $H \cong \mathbb{Z}_p^n$ be a generalized Fermat group of $S$ of type $(p,n)$. Since $M(S)= \mathbb{R}$, there is
an anticonformal automorphism $f:S\to S$. Since $f^{-1} H f$ is also a generalized Fermat group of type $(p,n)$.  As this is unique, up to conjugation by Theorem \ref{unico}, there is some $g \in \mbox{Aut}(S)$ so that
$(gf)^{-1}H (g f)=H$.

If we set $\tau : = fg$, then $\tau$ is an anticonformal automorphism of $S$ that   normalizes $H$. It follows that $\tau$  induces an anticonformal automorphism $\widehat{\tau}$ of the orbifold $\mathcal{O}=S/H$.

Since $\widehat{\tau}$ is
anticonformal, there exists $M \in \mathbb{N}$ so that $| \widehat{\tau}
|=2M$.  We set $N=M$ for $N$ odd and $N=2M$ for $N$ even.

As we are assuming $n$ even, then $n+1=2NA+NB+2C$ is odd, from which we obtain that $N$ is necessarily odd (that is, $N=M$). In this way, $\widehat{\tau}^{N}$ is an anticonformal involution and,
by Proposition \ref{involucion}, the Riemann surface $S$ is real.
\end{proof}

\s

\subsection{}

In the proof of the Theorem \ref{primo1}  we strongly uses the parity of $n$. For $n$  odd we have the following  parcial result.

\vspace{0,5 cm}

\begin{theo}
Let $(S,H)$ a generalized Fermat pais of type either $(p,3)$ or $(p,5)$ with $p>2$ prime. If the field of moduli of $S$ is ${\mathbb R}$, then $S$ is real.
\end{theo}
\begin{proof}
As we are assuming that $\mbox{M}(S)= \mathbb{R}$, there
exists an anticonformal automorphism  $f : S \rightarrow S$.  We have that $f^{-1}Hf$ is other generalized Fermat group of the same type, so by Theorem \ref{unico}, there is some $g \in \mbox{Aut}(S)$ so that
$(fg)^{-1}H(fg)=H$. If we set $\tau:=gh$, then $\tau$ is an anticonformal automorphism of $S$ that
normalizes $H$. In particular, $\tau$ induces an anticonformal automorphism $\widehat{\tau}$ of the orbifold
$S/H$.

As already noted, we may assume that  $\widehat{\tau}(z)=\mbox{e}^{i\theta}/\overline{z}$ with $\theta = 2 \pi /N$, where  $N\ge 1$. As $n+1=2NA+NB+2C$ and $n$ is odd, then $NB$ must be even.

 If $N \in \{1,2\}$, then $\widehat{\tau}$ has order two and, by Proposition \ref{involucion},
  it follows that $S$ is real. Now on, we assume $N \geq 3$.

\begin{enumerate}
\item Type $(p,3)$. \\
  If $N \ge 3$ odd, then $4=2NA + NB + 2C \ge 6A + 3B +2C$ and, as $C \in \{0,1\}$, this is not possible.

 If $N \ge 4$ is even, then $4=NB + 2C \ge 4B + 2C$. As $C \in \{0,1\}$, we must have that $N=4$, $B=1$ and $C=0$. Up to conjugation by a M\"obius transformation that keeps the unit disc invariant, we may assume that the cone points  are $\mu_{1} \in [1,+\infty)$, $\mu_{2}=it$, where $t \in (0,1]$, $\mu_{3}=-\mu_{1}$ and $\mu_{4}=-\mu_{2}$. In this case, the orbifold $S/H$ admits the anticonformal involution $\eta(z)=\overline{z}$. By Proposition \ref{involucion} we have that $S$ is real.

\vspace{0,5 cm}
\item  Type $(p,5)$.\\
If $N=3$ then $6 = 6A+3B+2C$, then $(A,B,C) \in \{ (1, 0, 0), (0, 2, 0)  \}$.

If $N \ge 4$ is even, then $A=0$ and $6=NB + 2C$, and this implies that either (a) $N=4$, $B=C=1$ or (b) $N=6$, $B=1$ and $C=0$.

If $N \ge 5$ is odd, then $6=2NA + NB + 2C \ge 10A + 5B + 2C$, and this implies that $A=0$. As $C \in \{0,1\}$, the quality $6=NB + 2C$ is not possible.

The following table summarizes the possible cases.

\vspace{0,5 cm}
\begin{center}
\begin{tabular}{ |c | c | c | c | }
\hline N & A & B & C \\
\hline 3 & 0 & 2 & 0 \\
\hline 3 & 1 & 0 & 0 \\
\hline 4 & 0 & 1 & 1 \\
\hline 6 & 0 & 1 & 0 \\
\hline
\end{tabular}
\end{center}
\vspace{0,5 cm}

\begin{enumerate}
\vspace{0,5 cm}
\item[(i)] If $(N, A, B, C)=(3,0,2,0)$, then $\widehat{\tau}$ has order six and there are two orbits: each one of length three. We can suppose that the orbits are $\{ 1, \omega, \omega^2 \}$ and $\{ \mu, \mu \omega, \mu \omega^2 \}$ with $\omega=e^{2 \pi i/3}$ and $\mu \in \mathbf{S}^1 - \{  1, \omega, \omega^2 \}$. This  configuration also admits the reflection in $\mathbf{S}^1$, $\eta(z)=1/ \overline{z}$, as an anticonformal involution. A  lifts of $\eta$ is an anticonformal automorphism of $S$ of order either $2$ or $2p$. As before, $S$ admits an anticonformal involution, so $S$ is real.
\vspace{0,5 cm}
\item[(ii)] If $(N, A, B, C)=(3,1,0,0)$, then $\widehat{\tau}$ has order six and there is an unique orbit of six elements. We can suppose that the orbit is $\{ \lambda,  \omega / \lambda, \omega^2 \lambda, 1/ \lambda,  \omega\lambda, \omega^2 / \lambda  \}$ with $\lambda > 1$ and $\omega=e^{2 \pi i/3}$.  It is clear that this points configuration also admits the conjugation as an anticonformal automorphism and, as above, $S$ is real.
\vspace{0,5 cm}
\item[(iii)] If $(N, A, B, C)=(4,0,1,1)$, then  $\widehat{\tau}$ has order four and there are two orbits: one of length two and one of length four. Without loss of generality, we can suposse that the orbits are $\{ 0, \infty \}$  y $\{ \lambda, i/ \lambda, - \lambda, -i/\lambda \}$, with $\lambda >1$ and $\widehat{\tau}(z)=i/ \overline{z}$. It is clear that this points configuration also admits the conjugation as an anticonformal involution and, as above, $S$ is real.\vspace{0,5 cm}
\item[(iv)] If $(N, A, B, C)=(6,0,1,0)$, then $\widehat{\tau}$ has order six and there is an unique orbit of six elements. We can suppose that the orbit is $\{ \lambda,  \omega \lambda, \omega^2 \lambda, -1/ \lambda,  -\omega/\lambda, -\omega^2 / \lambda \}$ with $\lambda \neq 1$ and $\omega=e^{2 \pi i/3}$.  It is clear that this points configuration also admits the conjugation as an anticonformal automorphism and, as above, $S$ is real.

\end{enumerate}
\end{enumerate}
\end{proof}


\begin{thebibliography}{99}


\bibitem{humbert}
A. Carocca, V. Gonz\'alez, R. Hidalgo and R. Rodr\'iguez.
Generalizated Humbert Curves,  {\it Israel Journal of Mathematics}
{\bf 64}, No. 1, (2008), 165-192.





\bibitem{Earle}
C.J. Earle.
On the moduli of closed Riemann surfaces with symmetries.
{\it Advances in the Theory of Riemann Surfaces} (1971), 119-130. Ed. L.V. Ahlfors et al.
(Princeton Univ. Press, Princeton).



\bibitem{Farkas-Kra}
H.M. Farkas and I. Kra.
{\it Riemann Surfaces}.
Graduate Texts in Mathematics {\bf 71}. Second edition (1991).


\bibitem{Fermat}
G. Gonz\'alez, R. Hidalgo, M. Leyton, {\it{Generalizated Fermat Curves}}, Journal of Algebra, vol. 321, 2009,1643 - 1660.


\bibitem{Hidalgo}
R.A. Hidalgo.
Non-hyperelliptic Riemann surfaces with real field of moduli but not definable over the reals.
{\it Archiv der Mathematik} {\bf 93} (2009), 219-222.

\bibitem{Maskit}
B. Maskit, Kleinean Groups, {\it Grundlerhren} Math. Wiss., vol.
287, Springer-Verlag, (1988).


\bibitem{Shimura}
G. Shimura.
On the field of rationality for an abelian variety.
{\it Nagoya Math. J.} {\bf 45} (1972), 167-178.

\bibitem{Silhol}
R. Silhol.
Moduli problems in real algebraic geometry.
{\it Real Algebraic Geometry} (1972), 110-119. Ed. M. Coste et al. (Springer-Verlag, Berlin).




\bibitem{Weil}
A. Weil.
The field of definition of a variety.
{\it  Amer. J. Math.} {\bf 78} (1956), 509-524.
\end{thebibliography}
\end{document}